\documentclass[11pt]{amsart}

\usepackage[french]{babel}
\usepackage{epigamathfr}

\numberwithin{equation}{section}

\usepackage{xypic}
\usepackage{dsfont}

\newtheorem{thm}{Th\'eor\`eme}[section]

\newtheorem{prop}[thm]{Proposition}
\newtheorem{lem}[thm]{Lemme}

\newtheorem{cor}[thm]{Corollaire}

\theoremstyle{definition}
\newtheorem{defi}[thm]{D\'efinition}

\theoremstyle{remark}
\newtheorem{rem}[thm]{Remarque}

\newenvironment{pro*}[1][Proof]{{\it{#1:}} }{}
\newenvironment{pro**}[1][]{{\it{#1}} }{\hfill $\square$}

\newcommand{\bF}{{\mathbb F}}

\newcommand{\bQ}{{\mathbb Q}}

\newcommand{\bZ}{{\mathbb Z}}

\newcommand{\dA}{{\mathcal A}}

\newcommand{\dO}{{\mathcal O}}
\newcommand{\dP}{{\mathcal P}}
\newcommand{\dQ}{{\mathcal Q}}

\DeclareMathOperator{\im}{im}

\DeclareMathOperator{\Spec}{Spec}

\DeclareMathOperator{\h}{H}

\EpigaVolumeYear{4}{2020}
\EpigaArticleNr{5}
\ReceivedOn{May 7, 2019}
\InFinalFormOn{November 27, 2019}
\AcceptedOn{March 28, 2020}

\title{Sur l'existence du sch\'ema en groupes fondamental}
\titlemark{Sch\'ema en groupes fondamental}

\author{Marco Antei}
\address{Escuela de Matematicas, CIMPA,
    Universidad de Costa Rica,
    Ciudad universitaria Rodrigo Facio Brenes,
    Costa Rica}
\email{marco.antei@ucr.ac.cr}
\author{Michel Emsalem}
\address{Laboratoire Paul Painlev\'e, U.F.R. de Math\'ematiques, Universit\'e des Sciences et des Technologies de Lille 1, 59 655 Villeneuve d'Ascq, France}
\email{emsalem@math.univ-lille1.fr}
\author{Carlo Gasbarri}
\address{IRMA, UMR 7501
 7 rue Ren\'e-Descartes
 67084 Strasbourg Cedex, France}
\email{gasbarri@math.unistra.fr}

\authormark{M. Antei, M. Emsalem et C. Gasbarri}

\AuthorsInFrench

\AbstractInEnglish{Let $S$ be a Dedekind scheme, $X$ a connected $S$-scheme locally of finite type and $x\in X(S)$ a section. The aim of the present paper is to establish the existence of the fundamental group scheme of  $X$, when $X$ has reduced fibers or when $X$ is normal. We also prove the existence of a group scheme, that we will call the quasi-finite fundamental group scheme of $X$ at $x$, which classifies all the quasi-finite torsors over $X$, pointed over $x$. We define Galois torsors, which play in this context a role similar to the one of Galois covers in the theory of \'etale fundamental group.}



\TitleInEnglish{On the existence of the fundamental group scheme}

\AbstractInFrench{Soient $S$ un sch\'ema de Dedekind, $X$ un $S$-sch\'ema connexe localement de type fini et $x\in X(S)$ une section. L'objet du pr\'esent papier est d'\'etablir l'existence du sch\'ema en groupes fondamental de $X$ lorsque $X$ est \`a fibres r\'eduites ou quand $X$ est normal. Sous des hypoth\`eses de normalit\'e, on d\'emontre aussi l'existence d'un sch\'ema en groupes, qu'on appellera sch\'ema en groupes fondamental quasi-fini de $X$ en $x$, qui classifie tous les torseurs quasi-finis au dessus de $X$, point\'es au dessus de $x$. On introduit les torseurs galoisiens, qui jouent dans ce contexte un peu le r\^ole des rev\^etements galoisiens dans la th\'eorie du groupe fondamental \'etale.}

\acknowledgement{Marco Antei a re\c{c}u le soutien du projet TOFIGROU (ANR-13-PDOC-0015-01).
Michel Emsalem a re\c{c}u le soutien du Labex CEMPI (ANR-11-LABX-01).
Carlo Gasbarri a re\c{c}u le soutien du projet ANR-16-CE40-0008, FOLIAGE.}

\begin{document}


\removeabove{0.5cm}
\removebetween{0.5cm}
\removebelow{0.5cm}

\maketitle

\begin{prelims}

\DisplayAbstractInFrench

\bigskip

\DisplayKeyWordsfr

\medskip

\DisplayMSCclassfr

\bigskip

\languagesection{English}

\bigskip

\DisplayTitleInEnglish

\medskip

\DisplayAbstractInEnglish

\end{prelims}


\newpage

\setcounter{tocdepth}{1} 

\tableofcontentsfr


\section{Introduction}

\subsection{Objectifs et plan du travail}

Soit $(X,x_0)$ un sch\'ema connexe point\'e. Le groupe fondamental $\pi_1^{\acute{e}t}(X, x_0)$ de $X$ classifie les rev\^etements \'etales finis de $X$, plus pr\'ecis\'ement il classifie les torseurs sur $X$ sous des sch\'emas en groupes finis \'etales : les rev\^etements galoisiens \'etales sont en bijection avec les quotients finis de  $\pi_1^{\acute{e}t}(X, x_0)$. Comme d\'ej\`a mentionn\'e dans SGA1, dans plusieurs situations naturelles, on voudrait avoir \`a sa disposition un sch\'ema en groupes qui classifie les torseurs sur $X$ sous des sch\'emas en groupes finis et plats (et non seulement \'etales) sur une base fix\'ee $S$ . On ne sait pas si un tel sch\'ema en groupes existe toujours. Quand on sait prouver son existence, on l'appellera le {\it sch\'ema en groupes fondamental} de $X$ sur $S$ et on le notera $\pi_1(X,x_0)$.

Pour construire le sch\'ema en groupes fondamental de $X$ sur $S$, on est confront\'e au probl\`eme suivant: on consid\`ere la cat\'egorie $\mathcal{P}(X)$ des triplets $(Y,G,y_0)$ o\`u $G$ est un sch\'ema en groupes fini et plat sur $S$, $Y$ est un $G$-torseur sur $X$ et $y_0$ est un $S$-point de $Y$ au dessus de $x_0$. Si cette cat\'egorie est cofiltr\'ee, alors le sch\'ema en groupes fondamental de $X$ sur $S$ existe (et il est la limite projective des sch\'emas en groupes $G$ qui interviennent dans la d\'efinition de la cat\'egorie). Le fait que cette cat\'egorie soit cofiltr\'ee n'est pas toujours facile \`a v\'erifier.

L'existence du sch\'ema en groupes fondamental pour des sch\'emas r\'eduits et connexes sur un corps a \'et\'e \'etablie dans l'article fondateur \cite{Nor2}.  L'\'etape suivante est d'\'etudier le cas o\`u la base $S$ est un sch\'ema de Dedekind. Soit $f:X\to S$ un morphisme muni d'une section $x_0\in X(S)$. Dans cet article on d\'emontre l'existence du sch\'ema en groupes fondamental $\pi_1(X,x_0)$ quand au moins une des deux hypoth\`eses suivantes est v\'erifi\'ee:

\smallskip

(A) $X\to S$ est s\'epar\'e, fid\`element plat, localement de type fini et pour tout point $s\in S$, la fibre $X_s$ est r\'eduite.
\smallskip

(B) $X\to S$ est s\'epar\'e, fid\`element plat, localement de type fini et pour tout point $x\in X\setminus X_{\eta}$, l'anneau local $\dO _x$ est int\'egralement clos, o\`u $\eta $ d\'esigne le point g\'en\'erique de $S$.
\smallskip

En particulier le sch\'ema en groupes fondamental de $X$ sur $S$ existe si $X$ est connexe et normal. Ceci est souvent suffisant pour les applications.
\medskip

On g\'en\'eralise ensuite cette \'etude \`a la question de l'existence du groupe fondamental sch\'ematique qui classifie les torseurs sous des sch\'emas en groupes {\it quasi-finis} sur $S$. On montre qu'il existe pour des sch\'emas qui v\'erifient la propri\'et\'e suivante:

\smallskip

(C) $X\to S$ est localement de type fini, s\'epar\'e, fid\`element plat, et pour tout point $s\in S$, $X_s$ est int\`egre et normal (en particulier $X$ est int\`egre et normal, comme on le rappellera en d\'ebut de la section \ref{fondamental}).
\medskip

Le plan de ce travail est le suivant : apr\`es des pr\'eliminaires techniques, on \'etudie dans la section \ref{reduction}, pour un $G$-torseur sur $X$, \`a quelle condition l'existence d'une r\'eduction de sch\'ema en groupes structural \`a un sous-sch\'ema en groupes ferm\'e de $G$ passe d'un ouvert dense de $X$ \`a $X$ lui-m\^eme. Ce sera la clef de la preuve dans la section \ref{fondamental} du fait que sous (A) ou (B), la cat\'egorie $\mathcal{P} (X)$ est cofiltr\'ee, et de la propri\'et\'e similaire pour les torseurs sous des sch\'emas en groupes quasi-finis. Dans la section \ref{galoisien}, on introduit les notions de torseurs galoisiens et quasi-galoisiens, outils qui nous permettront, sous certaines hypoth\`eses de comparer la restriction \`a un ouvert $U$ dense dans $X$ du sch\'ema en groupes fondamental de $X$ et du sch\'ema en groupes fondamental de $U$.

Dans l'article \cite{GAS}, le troisi\`eme auteur pr\'etendait avoir prouv\'e que la cat\'egorie $\mathcal{P}(X)$ est cofiltr\'ee quand $X$ est un sch\'ema irr\'eductible, r\'eduit et fid\`element plat sur $S$. Malheureusement dans cette preuve il y a une faute : le lemme 2.2 de op. cit. est faux, comme le montre le contre-exemple de J. Tong pr\'esent\'e au paragraphe \ref{sezJIL}. La construction du sch\'ema en groupes fondamental sous ces hypoth\`eses sur $X$ s'av\`ere \^etre encore un probl\`eme ouvert.

\section*{Remerciements}
Nous tenons \`a remercier Jilong Tong, qui nous a autoris\'e \`a pr\'esenter ici son contre-exemple, d\'ecrit \`a la section \ref{sezJIL}, H\'el\`ene Esnault  pour son int\'er\^et pour notre travail au travers de nombreux \'echanges et Matthieu Romagny pour sa contribution \`a la preuve du corollaire \ref{quotient-de-torseur}.

\section{Pr\'eliminaires}
\label{Prel}

\subsection{Notations}

Dans cet article $S$ d\'esignera un sch\'ema de Dedekind, i.e. un sch\'ema localement noeth\'erien, irr\'eductible et normal de dimension 0 ou 1, dont on notera par $\eta=\Spec(K)$ le point g\'en\'erique. Partout dans le texte un morphisme de sch\'emas sera dit quasi-fini s'il est de type fini et si chaque fibre est un ensemble fini de points. On se donne un $S$-sch\'ema s\'epar\'e, localement de type fini et fid\`element plat $X\to S$. Par l'expression  \emph{torseur fini (resp. quasi-fini)} sur $X$, on entendra un torseur sur $X$ au sens $fpqc$ sous l'action d'un $S$-sch\'ema en groupes fini et plat (resp. quasi-fini, affine et plat).

\subsection{Torseurs et sections}
\label{Preliminaires}

On utilisera tout le long de ce travail l'\'enonc\'e suivant qui assure la repr\'esentativit\'e du quotient d'un sch\'ema sous l'action d'un sch\'ema en groupes, sous certaines hypoth\`eses, \'enonc\'e pour lequel nous r\'ef\'erons \`a
\cite{Ana}, th\'eor\`eme 7, appendice 1. La compatibilit\'e du quotient (par l'action d'un sch\'ema en groupes de type fini) et du changement de base est assur\'ee par \cite[Expos\'e IV, 3.4.3.1]{SGA}.

\begin{thm}\label{teoQuozReno}
Soit $T$ un sch\'ema localement noeth\'erien. Soient $Z\to T$ un $T$-sch\'ema localement de type fini et quasi-fini, et $H\to T$ un sch\'ema en groupes affine plat  et quasi fini agissant sur $Z$.  Supposons que le morphisme naturel $Z\times_T H\to Z\times_T Z$ soit une immersion ferm\'ee.

Le faisceau $(Z/H)_{fpqc}$ est alors repr\'esent\'e par un sch\'ema $Z/H$.  De plus le morphisme canonique $Z\to Z/H$ est fid\`element plat et le morphisme canonique $Z \times _T H \to Z\times _{Z/H} Z $ est un isomorphisme, ce qui fait de $Z \to Z/H$ un $H$-torseur .
\end{thm}

\begin{cor}\label{quotient-de-torseur} Soit $S$ un sch\'ema de Dedekind, $G\to S$ un $S$--sch\'ema en groupes affine, plat et quasi-fini (resp. fini) et $H\subset G$ un sous-sch\'ema en groupes ferm\'e plat sur $S$.

Soit $X \to S$ un $S$-sch\'ema s\'epar\'e, localement de type fini et fid\`element plat et $Y \to X$ un torseur sous $G$.

Alors $Y \to Y/H$ est fid\`element plat et le morphisme $Y/H \to X$ issu de la factorisation $Y \to Y/H \to X$ est affine, plat et quasi-fini (resp. fini). En particulier $Y/H \to S$ est s\'epar\'e. \end{cor}

\proof Ces propri\'et\'es sont locales pour la topologie fppf. Pour ce qui concerne la fid\`ele platitude de $Y \to Y/H$, il suffit de montrer que $G \to G/H$ est fid\`element plat, ce qui est une cons\'equence du th\'eor\`eme \ref{teoQuozReno} et du fait que le morphisme $\mu : G\times _SH \to G\times _S G$, d\'efini par $\mu (g,h)=(g, gh)$ est une immersion ferm\'ee. La platitude de $Y/H \to X$ provient de la platitude de $Y \to X$ et de la fid\`ele platitude de $Y \to Y/H$. Pour le reste de l'\'enonc\'e, il suffit donc de montrer que $G/H \to S$ est affine et quasi-fini (resp. fini).

Le morphisme $G/H \to S$ est quasi-fini : il suffit de montrer que ce morphisme est localement de type fini. Cette propri\'et\'e est locale sur la source (\cite[Lemma 34.25.2, Tag 036O]{ST}). Puisque  $G\to S$ est quasi-fini et $G \to G/H$ est fid\`element plat, le morphisme $G/H \to S$ est quasi-fini.

Le morphisme $G/H \to S$ est  s\'epar\'e : en effet si l'on note $\Delta : G/H \to G/H \times _S G/H$ la diagonale, le diagramme suivant est cart\'esien
$$\xymatrix{
G \times _S H \ar[r]^\mu \ar[d] & G\times _S G \ar[d]\\
G/H \ar[r]_\Delta & G/H \times _S G/H \\
}$$ o\`u les fl\`eches verticales sont fid\`element plates et $\mu (g,h)=(g,gh)$. Comme $\mu$ est une immersion ferm\'ee, il en est de m\^eme de $\Delta$.

Le morphisme $G/H \to S$ est quasi-compact : $G \to S$ l'est et $G \to G/H$ est surjectif.

On peut donc appliquer le th\'eor\`eme principal de Zariski (\cite{ST}, Lemma 36.38.3, Tag 05K0), qui nous assure l'existence d'un diagramme commutatif
$$\xymatrix{
G/H \ar[rd] \ar[r]^j& T \ar[d]_\pi \\
&S\\
}$$ o\`u $\pi $ est fini et $j$ est une immersion ouverte. En se restreignant \`a un ouvert affine de $S$, on peut supposer $T$ affine de dimension $1$. D'apr\`es le lemme 32.37.2 de \cite[Tag 09N9]{ST}, $G/H$ est affine.

Enfin si $G$ est fini sur $S$, quitte \`a se restreindre \`a un ouvert affine de $S$, l'alg\`ebre de $G/H$ est contenue dans l'alg\`ebre de $G$, qui est finie par hypoth\`ese, et est donc aussi finie sur $S$. Donc $G/H$ est fini sur $S$.
\endproof

\begin{defi}
Supposons que l'on soit dans les hypoth\`eses du corollaire \ref{quotient-de-torseur}.  Soient $T\to X$ et $Y\to X$ respectivement un $H$--torseur et un $G$--torseur. On dira que $T\to X$ {\it est obtenu par r\'eduction de sch\'ema en groupes structural de $G$ \`a $H$} \`a partir du $G$-torseur $Y\to X$, s'il existe un diagramme commutatif 
$$\xymatrix{
T \ar[rd] \ar[r]^\iota& Y \ar[d] \\
&X\\
}$$
qui est de plus $H \hookrightarrow G$-\'equivariant. Notons que $\iota$ est automatiquement une immersion ferm\'ee. Dans cette situation, d'apr\`es le th\'eor\`eme \ref{teoQuozReno}, le quotient $Y/H$ est repr\'esent\'e par un sch\'ema et le morphisme $Y \to X$ se factorise \`a travers $Y/H \to X$. En termes cohomologiques, le torseur $Y\to X$ est obtenu par r\'eduction de sch\'ema en groupes structural de $G$ \`a $H$ si et seulement si sa classe dans $\h ^1(X,G)$ est contenue dans l'image de l'application naturelle d'ensembles point\'es $\h ^1(X,H)\to \h ^1(X,G)$.\end{defi}

\begin{lem}[cf. \cite{DG}, chapitre III, \S 4, n${}^\circ$~4, p. 373]\label{cartesien}
Soit $Y \to X$ un $G$-torseur sous un sch\'ema en groupes affine $G$ et $H \subset G$ un sous-sch\'ema en groupes ferm\'e. Le diagramme
$$\xymatrix{Y/H \ar[r] \ar[d] & X \ar[d]\\
BH \ar[r] &BG\\}$$
est $2$-cart\'esien, o\`u $BG$ et $BH$ d\'esignent les champs classifiants des $G$-torseurs et $H$-torseurs respectivement.
\end{lem}

Dans la situation pr\'ec\'edente, si on laisse tomber l'hypoth\`ese de quasi-finitude, $Y/H$ est un espace alg\'ebrique. C'est ainsi qu'il doit \^etre compris dans le lemme \ref{cartesien}.

Nous utiliserons le lemme pr\'ec\'edent sous la forme suivante :
\begin{lem}\label{cartesien1} Dans les hypoth\`eses du corollaire \ref{quotient-de-torseur},
il existe une correspondance biunivoque, stable par changement de base, entre sections $s$ de $Y/H$ au dessus de $X$ et r\'eductions $T \to X$ du sch\'ema en groupes structural de $G$ \`a $H$ pour le $G$-torseur $Y \to X$. Cette correspondance est donn\'ee par le diagramme cart\'esien
$$\xymatrix{T \ar[r] \ar[d] & Y \ar[d] \\
X\ar[r]_s & Y/H.\\}$$
\end{lem}

Voici un corollaire dont l'int\'er\^et s'\'eclairera \`a la section \ref{sezSGF}.

\begin{cor}\label{lemLEM22}  Supposons que l'on soit sous les hypoth\`eses du corollaire \ref{quotient-de-torseur}. Soit $X' \to S$ un morphisme fid\`element plat et localement de type fini. Soit $g : X' \to X$ un $S$-morphisme tel que $\dO _X \to g_ \ast \dO _{X'}$ soit un isomorphisme. Soit $Y'\to X'$ le $G$-torseur obtenu comme tir\'e par $g$ du torseur $Y$ sur $X$.

Alors $Y'\to X'$ est obtenu par r\'eduction de sch\'ema en groupes structural de $G$ \`a $H$ si et seulement si $Y\to X$ l'est.\end{cor}

\begin{proof}  Seulement une direction n\'ecessite une preuve. On dispose du diagramme cart\'esien

$$\xymatrix{
Y'/H \ar[r] \ar[d]_{\tilde {j'}}& Y/H \ar[d]^{\tilde j}\\
X' \ar[r]_g& X\\
}$$

D'apr\`es le lemme \ref{cartesien1}, l'hypoth\`ese se traduit par l'existence d'une section $X' \to Y' /H$ du morphisme canonique $\tilde {j'}$ qui, compos\'ee avec le morphisme naturel $Y'/ H \to Y/H$ donne un morphisme $s' :X' \to Y/H$ qui s'ins\`ere dans le diagramme commutatif

$$\xymatrix{
Y'/H \ar[r] \ar[d]_{\tilde {j'}}& Y/H \ar[d]^{\tilde j}\\
X' \ar[r]_g\ar[ru]^{s'}& X\\
}.$$

Le morphisme $s'$ induit un morphisme de $\dO_X$-faisceaux $\tilde j_*(\dO_{Y/H})\to g_*(\dO_{X'})\simeq \dO_{X}$. Du fait que $\tilde j$ est affine, on obtient une section $s:X\to Y/H$ (\cite[Lemma 28.11.5, Tag. 01SA]{ST}) telle que $s\circ g= s'$ et donc d'apr\`es le lemme \ref{cartesien1} l'existence d'un $H$-torseur $Z \to X$ contenu dans le $G$-torseur $Y \to X$. \end{proof}

Le corollaire \ref{lemLEM22} a l'interpr\'etation cohomologique suivante :  sous les hypoth\`eses du corollaire, soient $g^\ast: \h ^1(X, G)\to \h ^1(X',G)$, $\alpha :\h ^1(X,H)\to \h ^1(X,G)$  et $\alpha' :\h ^1(X',H)\to \h ^1(X',G)$ les applications naturelles, alors $\im (\alpha ) = (g^\ast )^{-1} (\im (\alpha '))$.

\subsection{Un lemme technique}

\begin{lem}\label{technique} Soit $S$ un sch\'ema de Dedekind. Soient $X\to S$ et $X'\to S$ respectivement un $S$--sch\'ema localement de type fini et fid\`element plat  et un $S$-sch\'ema localement de type fini. Soit $w:X' \to X$ un $S$-morphisme. Alors $w$ est un isomorphisme sous l'une des hypoth\`eses suivantes :
\begin{enumerate}
\item $X$ et $X'$ sont int\`egres,  $w$ est fini et on suppose qu'il existe un ouvert dense $U \subset X$ tel que la restriction de $w$ \`a $U'=w^{-1} (U)$ soit un isomorphisme de $U'$ sur $U$, et $X$ est normal en tous les points de $X\setminus U$;
\item $X$ et $X'$ sont int\`egres, $w$ est fini, $X$ est normal en tout point hors de la fibre g\'en\'erique, et la restriction $w_\eta $ de $w$ \`a la fibre g\'en\'erique est un isomorphisme $X'_\eta \simeq X_\eta$ ;
\item $X'\to S$ est plat, $w$ est fini et surjectif, la restriction $w_\eta $ de $w$ \`a la fibre g\'en\'erique est un isomorphisme $X'_\eta \simeq X_\eta$ et pour tout point $s\in S$, la fibre $X_s$ est r\'eduite.
\item $X$ et $X'$ sont int\`egres, $X'\to S$ est fid\`element plat, $w$ est quasi-fini et s\'epar\'e, on suppose qu'il existe un ouvert dense $U \subset X$ tel que la restriction de $w$ \`a $U'=w^{-1} (U)$ soit un isomorphisme de $U'$ sur $U$ et, pour tout point $s \in S$, la fibre $X_s$ est normale et int\`egre, $U_s \not=\emptyset$ et $w_s$ est un morphisme fini.
\end{enumerate}
\end{lem}

\proof
\leavevmode
\begin{enumerate}

\item \label{1} Paraphrasant la preuve du lemme 8.12.10.1 de \cite{EGAIV-3}, posons $\dA = w_\ast (\dO _{X'})$. Avec cette notation $X' \simeq \Spec \dA$. Pour tout $x\in X$, $\dO _{X,x} \subset \dA _x \subset R(X)$ o\`u $R(X)$ d\'esigne le corps des fonctions de $X$. Il s'agit de montrer que pour tout $x\in X$, $\dO _{X,x} = \dA _x$. Si $x \notin U$, $\dO _{X,x} $ est int\'egralement clos et $\dA _x$ est entier sur $\dO _{X,x}$, et donc $\dO _{X,x} = \dA _x$. Si $x \in U$, $w$ \'etant un isomorphisme de $w^{-1} (U)$ sur $U$, $\dO _{X,x} = \dA _x$ encore.

\item \label{2bis} Avec les notations pr\'ec\'edentes, il s'agit encore de prouver que $\dO _{X,x} = \dA _x$ pour tout $x\in X$. On se trouve donc dans la situation du point (1) o\`u le sch\'ema de Dedekind de base est $\Spec \dO _{S,s}$, si l'on a d\'esign\'e par $s$ l'image de $x$ dans $S$, du fait que le point g\'en\'erique $\eta $ est ouvert dans $\Spec \dO _{S,s}$.

\item \label{3}   La preuve est une adaption de celle du lemme 1.3 de \cite{WWO} . Il suffit de montrer que pour tout point ferm\'e $s\in S$, la restriction de $w$ au dessus de $\Spec \dO _{S,s}$ est un isomorphisme. On est ramen\'e au cas o\`u $S$ est le spectre d'un anneau de valuation discr\`ete  $R$ avec corps de fractions $K$ et id\'eal maximal engendr\'e par $\pi$. De la surjectivit\'e de $w$ on d\'eduit que sur la fibre sp\'eciale $w _s:X'_s\to X_s$ est surjectif ; on peut supposer $X$ et $X'$ affines, $X=\Spec(A)$ et $X'=\Spec(A')$ : les hypoth\`eses impliquent qu'on a une inclusion
$A\subset A'\subset A\otimes_R K$. Comme $X_s$ est r\'eduit, on en d\'eduit que $w_s^{\#}:\dO_{X_s}\hookrightarrow {(w_s)}_*(\dO_{X'_s})$ est injectif (\cite{EGAI}, corollaire 1.2.7) donc $\pi A'\cap A= \pi A$ , ce qui implique la conclusion dans ce cas.

\item \label{4}  Le morphisme $w: X' \to X$ se factorise en $w= w'' \circ w'$, o\`u $w' : X' \to X''$ est une immersion ouverte, $X''$ est int\`egre et $w'' :X'' \to X$ est fini (\cite[Lemma 36.38.3, Zariski's main theorem, Tag 05K0]{ST}). Par ailleurs l'hypoth\`ese implique que $X$ est normal (\cite[ Theorem 23.9]{Mat} et son corollaire), et le point (1) assure alors que $w''$ est un isomorphisme. On conclut que $w$ est une immersion ouverte. Pour tout point $s\in S$, l'immersion ouverte des fibres en $s$, $U'_s \subset X'_s$ se factorise en $U'_s \subset T_s \subset X'_s$, o\`u $U'_s \subset T_s$ est un ouvert dense de $T_s$ et $j_s :T_s \to X'_s$ est une immersion ferm\'ee. Le morphisme $t_s= w_s \circ j_s$ est fini, donc  propre, et l'image $t_s(T_s)$, qui est ferm\'ee et contient l'ouvert dense $U_s$, est $X_s$ tout entier. On en d\'eduit que $w_s$ est surjectif, ceci pour tout $s\in S$, d'o\`u on conclut que $w$ est surjectif. L'immersion ouverte $w$ qui est surjective est un isomorphisme.

\end{enumerate}
\endproof

\begin{rem}\label{technique-plus} Il ressort de la preuve du point (\ref{4}) que, si l'on supprime l'hypoth\`ese que pour tout point $s\in S$, $U_s \not=\emptyset$ et $w_s$ est fini, on conclut n\'eanmoins que $w$ est une immersion ouverte.
\end{rem}

\section{R\'eduction de sch\'ema en groupes structural}\label{reduction}

Dans tout le paragraphe on consid\`ere un sch\'ema de Dedekind $S$, un $S$--sch\'ema $X\to S$ s\'epar\'e, fid\`element plat et localement de type fini, et un sch\'ema en groupes affine plat $G\to S$.

On se donne  un ouvert dense $U \subset X$ dont on note $S'$ l'image dans $S$ par le morphisme fid\`element plat $X \to S$ ($S'$ est donc un ouvert contenant le point g\'en\'erique $\eta $ de $S$) et un sous-sch\'ema en groupes ferm\'e $H' \subset G_{S'}$ plat sur $S'$.   Soit $Y \to X$ un $G$-torseur et $Y_U\to U$ sa restriction \`a $U$.  La question g\'en\'erale est la suivante.

\smallskip
Supposons que $Y_U\to U$ admette une r\'eduction de sch\'ema en groupes structural $Z'\to U$ de $G_{S'}$ \`a $H' \subset G_{S'}$ (correspondant \`a une section $s': U \to Y_U/H'$), le torseur $Y \to X$ admet-il lui-m\^eme une r\'eduction de sch\'ema en groupes structural $Z\to X$ de $G$ \`a un sous-sch\'ema en groupes ferm\'e $H\subset G$ plat sur $S$ dont la restriction \`a $U$ est pr\'ecis\'ement $Z'\to U$?\
\smallskip

Au vu du lemme \ref{cartesien}, cela se traduit par l'existence d'une section $s: X \to Y/H$ rendant commutatif le diagramme suivant :

$$\xymatrix{
Y_U/H' \ar[r] &Y/H \\
U\ar[u]^{s'} \ar[r] &X\ar[u]_s\\
}.$$
Remarquons que si $X$ est int\`egre, $U$ est un ouvert sch\'ematiquement dense dans $X$, et la section $s$, si elle existe, est unique.

Observons encore que dans le diagramme pr\'ec\'edent, $H$ est n\'ec\'essairement l'adh\'erence sch\'ematique dans $G$ de la restriction $H'_\eta $ de $H'$ \`a la fibre g\'en\'erique.

\subsection{Le cas o\`u $X \to S$ est \`a fibres r\'eduites }\label{Reduit} On suppose ici que pour tout point $s\in S$, la fibre $X_s$ est r\'eduite (en particulier $X$ est r\'eduit, cf. \cite[prop. 9.5.9]{EGAI}).

\begin{prop}\label{lem22}
On suppose que pour tout point $s\in S$,  $X_s$ est r\'eduit et que $G$ un $S$-sch\'ema en groupes fini et plat. Soient $Y\to X$ un $G$-torseur et $H'$ un sous-sch\'ema ferm\'e en groupes de $G_{\eta}$. On suppose que $Y_{\eta}\to X_{\eta}$ admet une r\'eduction $Z\subset Y_\eta$ de sch\'ema en groupes structural \`a $H'$. Alors le $G$-torseur $Y \to X$ admet lui-m\^eme une r\'eduction $\overline Z \subset Y$ de sch\'ema en groupes structural \`a $\overline {H'}$, o\`u $\overline Z$ (resp. $\overline {H'}$) d\'esigne la cl\^oture sch\'ematique de $Z$ dans $Y$ (resp. de $H'$ dans $G$).

\end{prop}

\proof
La section $s : X_\eta \to Y_\eta /H'$ qui est une immersion ferm\'ee s'\'etend en une immersion ferm\'ee $s' :X' \to Y/\overline {H'}$ o\`u $X'$ d\'esigne la cl\^oture sch\'ematique de $X_\eta $ dans $Y / \overline {H'}$.

Le morphisme $u: X' \to X$, compos\'e du morphisme fini $Y/\overline {H'} \to X$ et de l'immersion ferm\'ee $s'$, est fini. La restriction \`a la fibre g\'en\'erique est un isomorphisme $u_\eta : X'_\eta \simeq X_\eta$. L'image $u(X')$ est un ferm\'e contenant $X_\eta $ qui est dense dans $X$, et donc $u$ est surjectif.
 On peut donc appliquer le point \ref{3} du lemme \ref{technique}, pour conclure que $u$ est un isomorphisme et donc l'existence d'une section $X \to Y/\overline {H'}$.

 De plus $\overline Z \subset Y$ est naturellement muni d'une action de $\overline {H'} $ compatible avec l'action de $G$ sur $Y$, provenant de l'action de $H'$ sur $Z$ compatible avec celle de $G_\eta $ sur $Y_\eta$, et des propri\'et\'es fonctorielles de la cl\^oture sch\'ematique. On en d\'eduit une immersion ferm\'ee $\overline Z /\overline {H'} \subset Y/\overline {H'}$, qui restreinte \`a la fibre g\'en\'erique co\" \i ncide avec $Z/H' =X_\eta \subset Y_\eta /H'$. Par ailleurs $\overline Z/ \overline {H'}$ est plat sur $S$, ce qui assure que $\overline Z/ \overline {H'} \subset Y/ \overline {H'}$ est la cl\^oture sch\'ematique $X' \simeq X$ de $X_\eta $ dans $Y/ \overline {H'}$. On aboutit au diagramme commutatif suivant (automatiquement cart\'esien) :
 $$\xymatrix{
 \overline Z \ar[r]\ar[d] &Y \ar[d] \\
 \overline Z /\overline {H'} =X \ar[r]& Y /\overline {H'}\\
 }$$
 o\`u les fl\`eches horizontales sont des immersions ferm\'ees, celle du haut \'etant $\overline {H'}$-\'equivariante. \endproof

\subsection{Le cas normal}
\label{Normal}
Dans cette section, inspir\'ee par \cite[Chapter II]{Nor2}, on \'etablira des \'enonc\'es similaires \`a la proposition \ref{lem22} sous des hypoth\`eses de normalit\'e.

\subsubsection{Le cas $G \to S$ fini}
\label{sezFINI}

Nous aurons besoin du lemme technique suivant.

\begin{lem}\label{technique2} Consid\'erons un carr\'e cart\'esien de sch\'emas
$$\xymatrix{
V \ar[r]^v \ar[d]_j &Z \ar[d]^{j'} \\
U \ar[r]_u &X\\
}$$
o\`u les fl\`eches horizontales sont des immersions ouvertes, $u(U)$ \'etant dense dans $X$, et les fl\`eches verticales des morphismes finis. On suppose $Z$ localement noeth\'erien et $X$ int\`egre et normal en tout point de $X \setminus u(U)$. Alors s'il existe une section $s:U \to V$ de $j$, il existe une section $s': X \to Z$ de $j'$ \'etendant la section $s$ et rendant le diagramme \'evident cart\'esien.
\end{lem}

\proof Posons $t=v \circ s$. C'est une immersion, qui est quasi-compacte d'apr\`es \cite[Lemma 27.5.3, Tag 01OX]{ST}. Il s'ensuit que $t$ est la compos\'ee $t=b\circ a$ de l'immersion ouverte $a: U \to \overline{t(U)}$ et de l'immersion ferm\'ee $b: \overline{t(U)} \hookrightarrow Z$ (\cite[Lemma 28.3.2, Tag 01QV]{ST}), o\`u $\overline{t(U)}$ est r\'eduit et irr\'eductible, car $U$ est int\`egre, et donc int\`egre. On pose $w=j' \circ b$ qui est fini et rend commutatif le diagramme
$$\xymatrix{
U \ar[r]^a \ar[d]_= & \overline {t(U)}\ar[d]^w \\
U \ar[r] &X\\
}.$$ L'existence de la section $s'$ est la cons\'equence du point 1 du lemme \ref{technique}. Le diagramme suivant est commutatif
$$\xymatrix{
U \ar[r]^s \ar[d]_u & V \ar[r]^j \ar[d]^v &U \ar[d]^u\\
X \ar[r]_{s'} & Y \ar[r]_{j'} &X\\
}.$$ Le grand rectangle et le carr\'e de droite sont cart\'esiens. Il s'ensuit que le carr\'e de gauche est cart\'esien.
\endproof

\begin{prop}\label{lemCARFIN} Soit $f:X\to S$ un $S$-sch\'ema fid\`element plat localement de type fini et int\`egre. Soit $U$ un ouvert dense de $X$, $S'=f(U)\subset S$ et supposons que $X$ soit normal en tous les points en dehors de $U$. Soit $G$ un $S$-sch\'ema en groupes fini et fid\`element plat et $Y \to X$ un $G$-torseur. Si la restriction \`a $U$ de $Y \to X$ admet une r\'eduction $T \subset Y_U$ de sch\'ema en groupes structural de $G_{S'}$ \`a $H'$ o\`u $H' \subset G_{S'}$ est un sous-sch\'ema en groupes ferm\'e fid\`element plat sur $S'$, le $G$-torseur $Y\to X$ admet lui-m\^eme une r\'eduction $\overline {T_\eta } \subset Y$ de sch\'ema en groupes structural de $G$ \`a $ \overline {H'_\eta }$, o\`u $\overline {T_\eta }$ (resp. $\overline {H' _\eta }$) d\'esigne la cl\^oture sch\'ematique de $T_\eta $ (resp. $H'_\eta $) dans $Y$ (resp. dans $G$).\end{prop}

\proof On remarque d'abord que par l'unicit\'e de la cl\^oture sch\'ematique, du fait que $H' \to S'$ est plat, la restriction \`a $S'$ de $\overline{ H'_\eta }$ est $H'$. Il y a une action de $(\overline {H'_\eta} )_X$ sur $Y$ compatible avec l'action de $G_X$ et qui, tir\'e en arri\`ere sur $U$, donne l'action de $H'$ sur $Y_U$. La compatibilit\'e du quotient et du changement de base entra\^ine que le diagramme suivant est cart\'esien
$$\xymatrix{
Y_U /H' \ar[r]^v \ar[d]_j &Y/\overline {H'_\eta }\ar[d]^{j'} \\
U \ar[r]_u &X\\
}.$$
Le lemme \ref{technique2} prouve l'existence d'une section de $j'$. On conclut la preuve comme celle de la proposition \ref{lem22}. \endproof

\subsubsection{Le cas $G \to S$ quasi-fini}

On suppose ici que $G\to S$ est affine et quasi-fini.

\begin{prop}\label{lemQF} Soit $f:X\to S$ un $S$-sch\'ema s\'epar\'e, fid\`element plat et localement de type fini. Soit $U$ un ouvert dense de $X$, $S'=f(U)\subset S$ et supposons que pour tout point $z\in S$, la fibre
$X_z$ soit int\`egre et normale (cf. la remarque \ref{remNorma}). Soit $G$ un $S$-sch\'ema en groupes affine et quasi-fini et plat. On se donne un $G$-torseur $Y \to X$. Supposons que la restriction \`a $U$ de $Y \to X$ admet une r\'eduction de sch\'ema en groupes structural \`a $H'$, o\`u $H' \subset G_{S'}$ est un sous-sch\'ema ferm\'e en groupes quasi-fini et fid\`element plat sur $S'$, que l'on notera $T \to U$.

Alors le $G$-torseur $Y \to X$ lui-m\^eme admet une r\'eduction $\overline {T_\eta } \subset Y$ de sch\'ema en groupes structural \`a $\overline {H'_\eta } \subset G$, o\`u $\overline {T_\eta } $ (resp. $\overline {H'_\eta }$) d\'esigne la cl\^oture sch\'ematique de $T_\eta $ (resp. $H'_\eta $ ) dans $Y$ (resp. $G$), si l'une des deux conditions suivantes est v\'erifi\'ee:

\begin{enumerate}
\item $S=S'$;
\item pour tout point ferm\'e $s \in S$, la restriction \`a $\Spec \dO _{S,s}$ de la compos\'ee $T_{\eta} \to Y_{\eta}\to  Y$ n'est pas une immersion ferm\'ee.
\end{enumerate}
 \end{prop}

\begin{proof}

D'apr\`es le th\'eor\`eme \ref{teoQuozReno} les quotients par des sch\'emas en groupes quasi-finis et plats existent ; on consid\`ere
le diagramme suivant:

 \begin{equation}\label{eqFIG1}
\xymatrix{Y_U/H'\ar[dd]_{j^{''}_U}\ar[rr]^u & & Y/\overline{H'_\eta }\ar[dd]^{j^{''}} \\
 & X'\ar[rd]^w\ar@{^{(}->}[ru]^b & \\
 U\ar@/_/@{-->}[uu]_s \ar[ru]^a \ar[rr]_i & & X }
  \end{equation}
 o\`u $u,a,i$ sont immersions ouvertes, $s,b$ sont des immersions ferm\'ees, $j''$ est quasi-fini et affine d'apr\`es le corollaire \ref{quotient-de-torseur}, $t:=u\circ s$ et $X':= \overline{t(U)}$ l'image sch\'ematique (muni de la structure de sch\'ema r\'eduit). Le sch\'ema $X'$ est r\'eduit et irr\'eductible, donc int\`egre. Le morphisme $w$ est affine et donc s\'epar\'e. Il ressort de la remarque \ref{technique-plus} que $w$ est une immersion ouverte.

Si $S=S'$, pour chaque point $s \in S$, $(Y/\overline{H'_\eta })_s\simeq Y_s/(\overline{H'_\eta })_s$ est un quotient par un sch\'ema en groupes fini, et d'apr\`es le corollaire \ref{quotient-de-torseur}, ${j''}_s$ et donc $w_s$ est fini. Les hypoth\`eses du point \ref{4} du lemme \ref{technique} sont v\'erifi\'ees, $w$ est un isomorphisme. Ceci prouve la proposition sous l'hypoth\`ese (1).

Si $S' \not=S$, d'apr\`es la premi\`ere partie de l'\'enonc\'e, on est dans la situation suivante : on dispose d'une section $s : X_{S'} \to Y_{S'} /H' $ du morphisme $Y_{S'} /H'  \to X_{S'}$, o\`u $H' \subset G_{S'}$ est un sous-sch\'ema ferm\'e en groupes quasi-fini et plat sur $S'$. Par unicit\'e, il suffit d'\'etendre cette section au dessus du spectre de l'anneau local de $S$ en tout $s \in S\setminus S'$. On est donc ramen\'e au cas o\`u $S = \mathrm{Spec} (R)$, $R$ \'etant un anneau de valuation discr\`ete, et $S' =\{ \eta  \} $, $U=X_{\eta}$, $T=T_{\eta}$ et $H'=H'_{\eta}$. On observe d'abord que $f\circ w(X')=S$. En effet le diagramme suivant est cart\'esien :

 $$\xymatrix{
 T \ar[r] \ar[d] &Y_{\eta} \ar[d] \ar[r] &Y \ar[d]\\
 T /H'' = X_\eta \ar[r]^s & Y_\eta /H' \ar[r]^u &Y/\overline {H'}  \\
 }$$
et le morphisme $T \to Y_{\eta}\to  Y$ est une immersion ferm\'ee si et seulement si $X_{\eta} \to Y_\eta /H' \to  Y/ \overline{H'}$ l'est aussi. Par d\'efinition ce dernier se factorise en   $X_{\eta} \to X' \to  Y/ \overline{H'}$. Puisque l'on a suppos\'e que $T \to Y_{\eta}\to  Y$ n'est pas une immersion ferm\'ee, comme $X' \to  Y/ \overline{H'_{\eta}}$ est une immersion ferm\'ee par d\'efinition, alors $a:X_{\eta} \to X'$ n'est pas une immersion ferm\'ee, et en particulier n'est pas un isomorphisme. Si l'image $f\circ w (X')$ \'etait $S'$, $w$ prendrait ses valeurs dans $X_\eta$, et $w\circ a$ serait l'identit\'e de $X_\eta$. Comme $w$ et $a$ sont des immersions ouvertes, cela impliquerait que ce sont des isomorphismes, ce qui est impossible. Donc $X'$ se surjecte sur $S$.

On est dans la situation suivante :
\begin{equation}\label{eqFIG5}
\xymatrix{ (Y/\overline{H' })_{X'}\ar[d]\ar[r] & Y/\overline{H' }\ar[dd]^{j^{''}} \\
  X'\ar[rd]^w\ar[ru]^b & \\
 & X }
 \end{equation}
 d'o\`u l'existence d'une section $b'$ de la premi\`ere fl\`eche verticale. Or, $(Y/\overline{H' })_{X'}\simeq Y_{X'}/\overline{H' }$ (\cite[Expos\'e IV, 3.4.3.1]{SGA}) o\`u $Y_{X'} \to X'$  est un $G$-torseur. Le $G$-torseur $Y_{X'} \to X'$ admet une r\'eduction de sch\'ema en groupes structural de $G$ \`a $\overline{H'}$, et d'apr\`es le point (1) de la proposition, il en est de m\^eme pour le $G$-torseur $Y \to X$ lui-m\^eme. Comme dans la preuve de la proposition \ref{lem22}, on conclut que la r\'eduction du sch\'ema en groupes structural de $G$ \`a $\overline {H'}$ de $Y \to X$ est $\overline {T_\eta } \subset Y$.

Il est clair par ailleurs que, comme la rectriction \`a $U$ de la section $b'$ est la section $s$ de d\'epart, la restriction \`a $U$ de la section $X \to Y/\overline{H' }$ \`a $U$ est aussi $s$. Cela signifie que la restriction \`a $U$ de la r\'eduction du sch\'ema en groupes structural de $G$ \`a $\overline {H'}$ obtenue est isomorphe \`a la donn\'ee de d\'epart $T \to U$. \end{proof}

 \begin{rem}\label{exMAU} Voici un exemple, qui explique pourquoi, dans la proposition \ref{lemQF}, on a introduit l'hypoth\`ese  que $T_\eta \hookrightarrow Y_\eta  \hookrightarrow Y $ n'est pas une immersion ferm\'ee. En g\'en\'eral l'adh\'erence sch\'ematique du torseur $T$ n'est pas forc\'ement un torseur sous l'adh\'erence sch\'ematique de $H'$ : on prend $X=S= \Spec (R)$ le spectre d'un anneau de valuation discr\`ete de caract\'eristique positive $p$, d'uniformisante $\pi $, $G= \Spec R[X]/(\pi ^{p-1} X^p -X)$, $H'$ r\'eduit \`a l'\'el\'ement neutre de $G_\eta $ et $T= \{ x=1/\pi \}$. Le sous-sch\'ema ferm\'e $T$ de $G_\eta $ est ferm\'e dans $G$ (pas de sp\'ecialisation en dehors de lui-m\^eme). Et $T$ n'est pas un torseur sous $\overline {H'}$ qui est r\'eduit \`a l'\'el\'ement neutre de $G$.
\end{rem}

Voici une version point\'ee de la proposition \ref{lemQF}.

\begin{prop}\label{lemQFpoint} Les hypoth\`eses sont celles de la proposition \ref{lemQF}. On suppose que $X$ est point\'e par $x\in X(S)$ et $Y$ par $y\in Y_x(S)$, et que $y_{S'}\in T(S')$. Alors $Y \to X$ admet une r\'eduction de sch\'ema en groupes structural de $G$ \`a $\overline {H'_\eta}$, $\overline {T_\eta } \to X$, point\'ee en $y$, et dont la restriction \`a $U$ est isomorphe \`a $T\to U$.
\end{prop}

\begin{proof} Il suffit de montrer que dans ces conditions, si $S' \not=S$, la seconde condition de la proposition \ref{lemQF} est v\'erifi\'ee. Apr\`es localisation en un point ferm\'e $s \in S$, $y_s\notin T_{\eta} $ et $y_s$ est dans l'adh\'erence de $T_\eta $ dans $Y$. Donc $T_\eta \hookrightarrow Y_{\eta}\to  Y$ n'est pas une immersion ferm\'ee. \end{proof}

\begin{cor}\label{generique} Les \'enonc\'es pr\'ec\'edents (propositions \ref{lemCARFIN} et \ref{lemQFpoint}) restent valables si l'on remplace l'ouvert $U$ par la fibre g\'en\'erique $X_\eta $.
\end{cor}
\proof
En effet, les r\'eductions des sch\'emas en groupes structuraux reviennent \`a montrer l'existence d'une section du morphisme naturel $Y/\overline{H'_\eta}\to X$, tout en sachant que cette section existe sur la fibre g\'en\'erique. Le lemme \ref{appendice} permet d'\'etendre cette section au dessus d'un ouvert de $S$ et de se ramener aux hypoth\`eses des propositions \ref{lemCARFIN} et \ref{lemQFpoint}.
\endproof

\begin{lem}\label{appendice} Soit $S$ un sch\'ema int\`egre et $\pi _1 : Z_1\to S$ et $\pi _2 : Z_2 \to S$ deux $S$-sch\'emas localement de pr\'esentation finie et s\'epar\'es. On suppose de plus que $Z_1$ et $Z_2$ sont quasi-compacts et que $\pi _2$ est ouvert. Soit $a: Z_1 \to Z_2$ un $S$-morphisme affine.

Si la restriction $a_\eta $ de $a$ \`a la fibre g\'en\'erique admet une section $s_\eta $, il existe un ouvert non vide $V$ de $S$ et une unique section $s_V$ de $a_V : Z_1 \times _S V \to Z_2 \times _S V$ dont la restriction \`a la fibre g\'en\'erique est $s_\eta $.
\end{lem}
\proof Le th\'eor\`eme 8.8.2 de \cite{EGAIV-3} s'applique \`a notre situation et assure l'existence d'un ouvert $V\subset S$ et d'un morphisme $s_V : Z_{2,V} \to Z_{1,V}$ prolongeant $s_\eta $. Comme $\pi _2$ est s\'epar\'e, $s_V$ est une section de $a_V$. L'unicit\'e provient de la s\'eparation de $\pi _1$ et du fait que la fibre g\'en\'erique est dense dans $Z_2$.
\endproof

\section{Le sch\'ema en groupes fondamental}\label{fondamental}
\label{sezSGF}
 L'objectif est de d\'emontrer qu'une certaine cat\'egorie de torseurs point\'es sur un certain $S$-sch\'ema de base $X\to S$ sous certains $S$-sch\'emas en groupes est cofiltr\'ee. On rappelle les conditions portant sur un $S$-sch\'ema $X$ connexe qu'on sera amen\'e \`a consid\'erer :
 \smallskip

(A) $X\to S$ est localement de type fini, s\'epar\'e, fid\`element plat, et pour tout point $s\in S$, $X_s$ est r\'eduit ;
\smallskip

(B) $X\to S$ est localement de type fini, s\'epar\'e, fid\`element plat, int\`egre et normal ;
\smallskip

(C) $X\to S$ est localement de type fini, s\'epar\'e, fid\`element plat, et pour tout point $s\in S$, $X_s$ est int\`egre et normal.
\medskip

\begin{rem}\label{remNorma} L'hypoth\`ese (C) implique l'hypoth\`ese (B): du fait que la fibre g\'en\'erique $X_{\eta}$ est int\`egre on d\'eduit ais\'ement l'int\'egrit\'e de $X$. Ensuite la normalit\'e de $S$ et de chaque fibre entra\^ine la normalit\'e de $X$ (cf. \cite[Theorem 23.9]{Mat} et son corollaire).
\end{rem}
\medskip

R\'esumons les r\'esultats obtenus dans la section \ref{reduction}. On suppose que $X \to S$ satisfait (A) ou (B) (resp. (C)). Soient $G \to S$ un sch\'ema en groupes plat et fini (resp. affine et quasi-fini) et $Y \to X$ un $G$-torseur. Si la fibre g\'en\'erique $Y_\eta \to X_\eta $ admet une r\'eduction \`a un $H'$-torseur $Z' \subset Y_\eta $, o\`u $H' \subset G_\eta $ est un sous-sch\'ema en groupes ferm\'e, alors la cl\^oture sch\'ematique $\overline {H'} $ de $H' $ dans $G$ est un sous-sch\'ema en groupes ferm\'e, plat et fini (resp. affine et quasi-fini), et la cl\^oture sch\'ematique $\overline {Z'} \subset Y$ est une r\'eduction du $G$-torseur $Y \to X$ \`a un $H'$-torseur. De plus, si les torseurs sont point\'es, il en est de m\^eme de leur r\'eduction

\medskip

Si $X\to S$ est un morphisme fid\`element plat et $x\in X(S)$ est un point fix\'e, alors on d\'efinit, dans l'esprit de \cite{GAS},
la cat\'egorie  $\mathcal{P}(X)$ (resp. $\mathcal{Q}f(X)$)  des torseurs sous l'action d'un sch\'ema en groupes fini (resp. quasi-fini) et plat,
point\'es au dessus de $x$. Chaque objet est donc un triplet $(Y,G,y)$ o\`u $G$ est un $S$-sch\'ema en groupes fini (resp. quasi-fini) et plat,
$Y$ est un $G$-torseur point\'e en $y\in Y_x(S)$.  Un morphisme de triplets $\varphi: (Y,G,y)\to (Z,H, z)$ est la donn\'ee d'un morphisme de sch\'emas en groupes $G\to H$ et un morphisme de $X$- sch\'emas point\'es $Y\to Z$ tels que le diagramme suivant soit commutatif
$$\xymatrix{G\times Y\ar[d]\ar[r]&Y\ar[d]\\
H\times Z\ar[r] & Z.}$$
On montre le th\'eor\`eme suivant:

\begin{thm}\label{thmSGF}
Soit  $X$ un sch\'ema connexe, localement de type fini et fid\`element plat sur $S$. Soit $x\in X(S)$ un point.
\begin{enumerate}
\item si l'une des conditions (A) ou (B) est v\'erifi\'ee, la cat\'egorie $\mathcal{P}(X)$ est cofiltr\'ee ; de plus dans $\mathcal{P}(X)$ les produits fibr\'es finis existent ;
\item si la condition (C) est v\'erifi\'ee, la cat\'egorie  $\mathcal{Q}f(X)$ est cofiltr\'ee ; de plus dans $\mathcal{Q}f(X)$ les produits fibr\'es finis existent ;
\item sous l'hypoth\`ese (C), le foncteur naturel $\mathcal{P}(X)\hookrightarrow \mathcal{Q}f(X)$ est pleinement fid\`ele et commute aux produits fibr\'es finis.

\end{enumerate}
\end{thm}
\proof
(1). On pr\'esente ici la preuve sous l'hypoth\`ese (A); sous l'hypoth\`ese (B) la preuve est similaire \`a la lumi\`ere des r\'esultats obtenus \`a la section \ref{sezFINI}. Puisque $\mathcal{P}(X)$ a un objet final (\`a savoir $(X,\{1\}_S, x)$) et les produits fibr\'es sur $X$ existent dans $\mathcal{P}(X)$,
il suffit de d\'emontrer que  pour trois objets quelconques $(Y_i, G_i, y_i), i=0, 1 , 2$ de $\mathcal{P}(X)$ et deux morphismes
$\varphi_i:(Y_i, G_i, y_i)\to (Y_0, G_0, y_0), i=1, 2$, il existe un quatri\`eme objet $(Y_3, G_3, y_3)$ et deux morphismes
$\psi_i:(Y_3, G_3, y_3)\to (Y_i, G_i, y_i), i=1, 2$ qui cl\^oturent le carr\'e. Soient $X_{\eta}$ et
$(Y_{i,\eta}, G_{i,\eta}, y_{i,\eta}), i=0, 1, 2$ les fibres g\'en\'eriques de $X$ et $(Y_i, G_i, y_i)$ respectivement. Le sch\'ema $X_\eta $ \'etant connexe, r\'eduit et muni d'un point rationnel, d'apr\`es
\cite[Chapter II, Propositions 1 and 2]{Nor2}, $Y_{1,\eta}\times_{Y_{0,\eta}}Y_{2,\eta}$ est un $G_{1,\eta}\times_{G_{0,\eta}}G_{2,\eta}$-torseur au dessus de $X_{\eta}$
(point\'e en $y_{1,\eta}\times_{y_{0,\eta}}y_{2,\eta}$).
D'apr\`es la proposition \ref{lem22}, on construit le $G_3$-torseur $Y_3\hookrightarrow Y_1 \times _X Y_2$ en prenant la cl\^oture du
$G_{1,\eta}\times_{G_{0,\eta}}G_{2,\eta}$-torseur $Y_{1,\eta}\times_{Y_{0,\eta}}Y_{2,\eta}$ dans le $G_1\times_S G_2$-torseur
$Y_1\times_X Y_2$. Les restrictions $\psi _i$ \`a $Y_3$ des projections $Y_1 \times _X Y_2 \to Y_i$, $i=1,2$, v\'erifient
$\varphi _1 \circ \psi _1 = \varphi _2 \circ \psi _2$ par fonctorialit\'e de la cl\^oture sch\'ematique. Finalement le point $y_3$ est, bien s\^ur, l'adh\'erence de $y_{1,\eta}\times_{y_{0,\eta}}y_{2,\eta}$ dans $Y_3$. Il est clair que l'objet $(Y_3, G_3,y_3)$ est le produit fibr\'e dans la cat\'egorie $\mathcal{P}(X)$ de $(Y_1,G_1,y_1)$ et $(Y_2,G_2,y_2)$ au dessus de $(Y_0,G_0,y_0)$. La preuve du point (2) cas quasi-fini est semblable, en utilisant la proposition \ref{lemQFpoint}. Enfin la derni\`ere assertion (3) provient du caract\`ere fonctoriel de la cl\^oture sch\'ematique. \endproof

De la conclusion du th\'eor\`eme \ref{thmSGF} il r\'esulte que les pro-objets des cat\'egories $\mathcal{P}(X)$ et $\mathcal{Q}f(X)$ sont repr\'esentables par des sch\'emas (\cite[prop. 8.2.3]{EGAIV-3}).  On notera $Pro-\mathcal{P}(X)$  et $Pro-\mathcal{Q}f(X)$ ces cat\'egories de pro-objets. On d\'eduit du th\'eor\`eme \ref{thmSGF} l'existence de pro-objets universels : un triplet $(\widehat{X},\pi_1(X,x),\widehat{x})$ (resp. $(\widehat{X}^{\text{qf}},\pi_1(X,x)^{\text{qf}},\widehat{x}^{\text{qf}})$) objet de $Pro-\mathcal{P}(X)$ (resp. objet de $Pro-\mathcal{Q}f(X)$) plats sur $S$, avec la propriet\'e universelle suivante: pour tout triplet $(P,G,x)$ objet de $Pro-\mathcal{P}(X)$ (resp. de $Pro-\mathcal{Q}f(X)$), il existe un unique morphisme $(\widehat{X},\pi_1(X,x),\widehat{x})\to (P,G,x)$ dans  $Pro-\mathcal{P}(X)$ (resp.$(\widehat{X}^{\text{qf}},\pi_1(X,x)^{\text{qf}},\widehat{x}^{\text{qf}})\to (P,G,x)$ dans $Pro-\mathcal{Q}f(X)$). Les objets de $\mathcal{P}(X)$ (resp. $\mathcal{Q}f(X)$) sont en correspondance bijective avec les $S$-morphismes $\pi_1(X,x) \to G$ (resp. $\pi_1(X,x)^{\text{qf}} \to G$) o\`u $G\to S$ est un sch\'ema en groupes fini (resp. quasi-fini) et plat, le torseur correspondant \`a un morphisme  $\varphi :\pi_1(X,x) \to G$ (resp. $\varphi :\pi_1(X,x)^{\text{qf}} \to G$) \'etant le produit contract\'e $\widehat X \wedge ^{\pi _1 (X,x)}G$ (resp. $\widehat X^{qf} \wedge ^{\pi _1 (X,x)}G$) \`a travers $\varphi $.

\begin{defi} Soit $X$ un sch\'ema connexe localement de type fini et fid\`element plat sur $S$.

\begin{enumerate} \item On suppose qu'une des deux conditions (A) ou (B) est v\'erifi\'ee. Alors le triplet  $(\widehat{X},\pi_1(X,x),\widehat{x})$ est appel\'e {\it triplet universel} de  $Pro-\mathcal{P}(X)$, le $S$-sch\'ema en groupes $\pi_1(X,x)$  est appel\'e le {\it sch\'ema en groupes fondamental de $X$}  au point $x$ et $\widehat{X}\to X$ est  appel\'e le {\it $\pi_1(X,x)$-torseur universel}.

\item On suppose la condition (C) v\'erifi\'ee. Alors le triplet  $(\widehat{X}^{\text{qf}},\pi_1(X,x)^{\text{qf}},\widehat{x}^{\text{qf}})$ est appel\'e {\it triplet universel} de  $Pro-\mathcal{Q}f(X)$, le $S$-sch\'ema en groupes $\pi_1(X,x)^{\text{qf}}$  est appel\'e le {\it sch\'ema en groupes fondamental quasi-fini de $X$}  au point $x$ et $\widehat{X}^{\text{qf}}\to X$ est  appel\'e le {\it $\pi_1(X,x)^{\text{qf}}$-torseur universel}.
\end{enumerate}
\end{defi}

 \section{Torseurs galoisiens}\label{galoisien}
Dans ce paragraphe $X$ d\'esigne un sch\'ema connexe localement de type fini et plat sur un sch\'ema de Dedekind $S$, point\'e par une section $x \in X(S)$, qu'on suppose v\'erifier l'une des conditions $(A)$, $(B)$ ou $(C)$.

 \begin{defi}  On dira qu'un objet des cat\'egories $Pro-\dP (X)$ et $ Pro-\dQ f (X)$ est {\it g\'en\'eriquement fini} si sa fibre g\'en\'erique est finie. De m\^eme un $S$-sch\'ema en groupes affine sera dit {\it g\'en\'eriquement fini} si sa fibre g\'en\'erique est finie.

Un objet g\'en\'eriquement fini $(T,M,t)$ de $Pro-\dP (X)$ est dit {\it galoisien} (resp. {\it quasi-galoisien}) si, pour tout morphisme $(g, \varphi ) : (Z,Q,z) \to (T,M,t)$ d'objets g\'en\'eriquement finis de $Pro-\dP (X)$, $\varphi $ est fid\`element plat (resp. sch\'ematiquement dominant). Un objet g\'en\'eriquement fini $(T,M,t)$ de $Pro-\dQ (X)$ est dit {\it qf-galoisien} (resp. {\it quasi-qf-galoisien}) si, pour tout morphisme $(g, \varphi ) : (Z,Q,z) \to (T,M,t)$ d'objets g\'en\'eriquement finis de $Pro-\dQ f (X)$, $\varphi $ est fid\`element plat (resp. sch\'ematiquement dominant).
 \end{defi}

 \begin{lem}\label{factorisationHai} Tout morphisme $\varphi : H \to G$ entre $S$-sch\'emas en groupes affines et plats sur un sch\'ema de Dedekind $S$ admet une factorisation, unique \`a isomorphisme pr\`es,
 $$\xymatrix{
 H \ar[r]^\varphi \ar[d]_\gamma & G\\
 G'' \ar[r]_\beta & G' \ar[u]_\alpha \\
 }$$
 o\`u $G'$ et $G''$ sont des sch\'emas en groupes affines et plats sur $S$, $\alpha$ est une immersion ferm\'ee, $\beta$ un morphisme mod\`ele (i.e. un isomorphisme sur la fibre g\'en\'erique) et $\gamma $ fid\`element plat. Si $ \varphi $ est sch\'ematiquement dominant, $\alpha $ est un isomorphisme. Si $\varphi $ est fid\`element plat, $\alpha $ et $\beta $ sont des isomorphismes.

Si $H$ est quasi-fini (resp. fini), il en est de m\^eme de $G''$. Si $G$ est quasi-fini (resp. fini) il en est de m\^eme de $G'$ et $G''$ est g\'en\'eriquement fini.

 \end{lem}

  \proof On se ram\`ene au cas  du spectre $S = \Spec (R)$ d'un anneau de Dedekind. Le morphisme $\varphi $ correspond \`a un morphisme d'alg\`ebres de Hopf $\tilde \varphi : RG \to RH$, dont l'image est une alg\`ebre de Hopf qu'on \'ecrit sous la forme $RG' \subset RH$. Le morphisme $\alpha :G' \to G$ correspondant \`a la surjection $RG \to RG'$ est une immersion ferm\'ee. On consid\`ere ensuite le satur\'e $KG' \cap RH$ de $RG'$ dans $RH$, qui est aussi une alg\`ebre de Hopf que l'on note $RG''$. Toutes ces alg\`ebres de Hopf sont sans $R$-torsion, donc plates sur $R$. D'apr\`es le th\'eor\`eme 4.1.1 de \cite{Duong-Hai}, le morphisme $\gamma : H \to G''$ correspondant \`a l'inclusion $RG'' \subset RH$ est fid\`element plat. Et par construction $\beta : G'' \to G'$ est un morphisme mod\`ele (en particulier sch\'ematiquement dominant). L'unicit\'e est claire, et le reste de la premi\`ere partie de l'\'enonc\'e s'en d\'eduit.

Si $H$ est quasi-fini, du fait que $H \to G''$ est surjectif, pour v\'erifier que $G''$ est quasi-fini, il suffit de voir que $G'' \to S$ est localement de type fini. Cette propri\'et\'e \'etant locale sur la source (\cite[Lemma 34.25.2, Tag 036O]{ST}), comme $H \to G'' $ est fid\`element plat, et $H \to S$ est localement de type fini, il en est de m\^eme de $G'' \to S$. Le reste de l'\'enonc\'e est clair.  \endproof

 \begin{lem}\label{2} Soient $\varphi : G_1 \to G_2$ et $\psi : G_2 \to G_3$ des morphismes de sch\'emas en groupes affines et plats sur un sch\'ema de Dedekind $S$. Si $\psi \circ \varphi $ est fid\`element plat (resp. sch\'ematiquement dominant), il en est de m\^eme de $\psi$.
 \end{lem}

 \proof On consid\`ere les morphismes d'alg\`ebres de Hopf correspondants : $\tilde \varphi : RG_2 \to RG_1$ et $\tilde \psi : RG_3 \to RG_2$. Par hypoth\`ese, $v= \tilde \varphi \circ \tilde \psi$ est injectif, ce qui se traduit par $\ker (\tilde \varphi ) \cap \im (\tilde \psi )=0$. Le morphisme $\tilde \psi$ est injectif. D'apr\`es le th\'eor\`eme 4.1.1 de \cite{Duong-Hai}, il suffit de montrer que l'image de $RG_3 $ par $\tilde \psi$ est satur\'ee dans $RG_2$. Soit donc $f \in RG_2$, $a \in R\setminus 0$, et $g \in RG_3$ tels que $af= \tilde \psi (g)$. On en d\'eduit $a\tilde \varphi (f)= v(g)$ et il existe donc $g' \in RG_3$ tel que $\tilde \varphi (f)= v(g')= \tilde \varphi \circ \tilde \psi (g')$. On en d\'eduit que $\tilde \varphi ( f-\tilde \psi (g'))=0$ et donc $f-\tilde \psi (g') \in \ker \tilde \varphi $. Finalement $af- \tilde \psi (ag ')= \tilde \psi (g-ag')\in \ker \tilde \varphi \cap \im \tilde \psi$ et donc $af- \tilde \psi (ag ')=0$ dans $RG_2$ qui est sans $R$-torsion. On en d\'eduit que $f= \tilde \psi (g')$ est dans l'image de $\tilde \psi$ ce qu'il fallait d\'emontrer.
 \endproof

 \begin{cor}\label{proj} Le torseur universel $\widehat X_x$ (resp. $\widehat X _x ^{qf}$) est limite projective de torseurs point\'es finis (resp. quasi-finis) quasi-galoisiens (resp. quasi-qf-galoisiens), et limite projective de torseurs g\'en\'eriquement finis galoisiens.
 \end{cor}

\proof On applique le lemme \ref{factorisationHai} \`a $H = \pi _1 (X,x)$ (resp. $H =\pi _1^{qf} (X,x)$), compte tenu du lemme \ref{2}. \endproof

\begin{prop}\label{propSURFIN} Soit $X$ un sch\'ema connexe fid\`element plat sur $S$ point\'e en $x\in X(S)$.  On suppose qu'une des conditions (A) ou (B) est v\'erifi\'ee (resp. (C) est v\'erifi\'ee). Soient $U$ un ouvert non vide de $X$, $S':=f(U)$ un ouvert de $S$; on suppose que $x_{S'}\in U$ et que pour tout point $x\in X\setminus U$, l'anneau local $\dO _x$ est int\'egralement clos. Alors $\pi_1(U,x_{S'})  \twoheadrightarrow  \pi_1(X,x)_{S'}  $ (resp. $\pi_1^{qf}(U,x_{S'})  \twoheadrightarrow  \pi_1^{qf}(X,x)_{S'}  $) est  sch\'ematiquement dominant (c.-\`a-d. le morphisme dual sur les alg\`ebres est injectif). En particulier sous l'hypoth\`ese (A) ou (B),  $\pi_1(X_{\eta},x_{\eta}) \twoheadrightarrow  \pi_1(X,x)_\eta   $ est fid\`element plat et, sous l'hypoth\`ese (C), $\pi_1(X_{\eta},x_{\eta}) \twoheadrightarrow  \pi_1^{qf}(X,x)_\eta   $ et $\pi_1^{qf}(X,x)_\eta  \twoheadrightarrow  \pi_1(X,x)_\eta  $ sont fid\`element plats.
\end{prop}
\proof Compte tenu du corollaire \ref{proj}, on est amen\'e \`a montrer que si $(Y,y) \to (X,x)$ est un $G$-torseur fini (resp. quasi-fini), quasi-galoisien, le morphisme naturel $\widehat U _{x_{S'}} \to Y_U$ (resp. $\widehat U _{x_{S'}}^{qf} \to Y_U$) est sch\'ematiquement dominant. Soit $\varphi :\pi _1( U,x_{S'}) \to G_{S'}$ (resp. $\varphi :\pi _1^{qf}( U,x_{S'}) \to G_{S'}$) le morphisme correspondant au $G_{S'}$-torseur $Y_U \to U$. D'apr\`es le lemme \ref{factorisationHai}, il se factorise en $\varphi = \alpha \circ \lambda $, o\`u $\alpha : G' \to G_{S'}$ est une immersion ferm\'ee et $\lambda $ est sch\'ematiquement dominant. Le torseur $Y_U \to U$ admet donc une r\'eduction de sch\'ema en groupes structural en un $G'$-torseur point\'e $Z \to U$. Il r\'esulte des propositions \ref{lemCARFIN} et \ref{lemQFpoint} que le $G$-torseur $Y \to X$ admet lui-m\^eme une r\'eduction de sch\'ema en groupes structural en un $\overline {G' _\eta}$-torseur, o\`u $\overline {G' _\eta}$ d\'esigne l'adh\'erence sch\'ematique de $G'_\eta $ dans $G$, dont la restriction \`a $U$ est isomorphe \`a $Z \to U$. Le fait que $Y \to U $ est quasi-galoisien entraine que $\overline {G'_\eta   }\simeq G$ et que $\overline Z \simeq Y$ comme $G$-torseurs. Il en r\'esulte que $Z \simeq Y_U$ comme $G_{S'}$-torseurs et que $\alpha $ est un isomorphisme. Donc $\varphi $ est sch\'ematiquement dominant, ce qu'il fallait d\'emontrer. Le reste de l'\'enonc\'e est clair. \endproof

\begin{prop}\label{propFIBGEO1} Soient $X$ et $X'$ deux $S$-sch\'emas connexes localement de type fini et fid\`element plats sur $S$, v\'erifiant une des conditions (A) ou (B). Soit $g:X'\to X$ un $S$-morphisme de sch\'emas tel que $\dO_X\to g_* (\dO_{X'})$ est un isomorphisme. On suppose l'existence d'une section $x'\in X'(S)$ et on note $x:=g(x')$, alors $\pi_1(X',x') \twoheadrightarrow  \pi_1(X,x)   $ est sch\'ematiquement dominant. C'est le cas en particulier si $g : X' \to X$ un $S$-morphisme propre, plat avec fibres g\'eom\'etriquement connexes et r\'eduites.
\end{prop}
\proof La preuve est la m\^eme que celle de la proposition \ref{propSURFIN} en utilisant le corollaire \ref{lemLEM22} \`a la place des propositions \ref{lemCARFIN} et \ref{lemQFpoint}.\endproof

\section{Appendice : Le contre-exemple de Jilong Tong}
\label{sezJIL}

La construction du sch\'ema en groupes fondamental d'un $S$-sch\'ema dans  \cite{GAS} repose sur l'\'enonc\'e suivant (lemme 2.2 de \cite{GAS}):

\smallskip

{\it Soit $S$ un sch\'ema de Dedekind de point g\'en\'erique $\eta=\Spec(K)$ et $G$ un sch\'ema en groupes fini et plat sur $S$. Soit $H\hookrightarrow G$ un sous sch\'ema en groupes. Soit $X\to S$ un $S$--sch\'ema fid\`element  plat {\it r\'eduit} et irreductible. Soit $P\to X$ un $G$--torseur et $Y\subset P_\eta$ un sous-sch\'ema ferm\'e qui est un $H_\eta$-torseur.  Alors la cl\^oture sch\'ematique $\overline Y$ de $Y$ dans $P$ est un $H$-torseur.}

\smallskip

L'exemple suivant montre que les hypoth\`eses sur $X$ dans l'\'enonc\'e pr\'ec\'edent sont insuffisantes. Soient
$\zeta $ une racine primitive $p$-i\`eme de $1$ ($p$ est un nombre premier), $X = \Spec \bZ [p\zeta ]= \Spec \bZ [V] /(p^{p-1}\Phi _p (V/p))$ o\`u $\Phi _p (T)= {\frac{T^p -1}{T-1}}=T^{p-1} + \dots +1$. Soit $P\to X$ le $\mu _p$-torseur trivial $P:=X \times \mu _p = \Spec \bZ [V,T]/( p^{p-1}\Phi _p (V/p),$ $T^p -1)$. On consid\`ere la fl\`eche diagonale $X_{\bQ}\simeq \mu _p \setminus \{ 1 \}  \to (X \times \mu _p) _{\bQ} = ((\mu _p \setminus \{ 1 \} ) \times \mu _p) _{\bQ}$. C'est un morphisme de torseurs pour le morphisme de sch\'emas en groupes $1 \to \mu _p$. On note $\overline Y$ la cl\^oture sch\'ematique de l'image ; alors $\overline Y= \Spec (A)$ o\`u la $\bZ $-alg\`ebre plate $A$ figure dans le diagramme suivant :

$$\xymatrix{\bZ [V,T] /( p^{p-1}\Phi _p (V/p), T^p -1) \ar[r] & A= \bZ [W] /(\Phi _p (W)) \\
\bZ [V]/ (p^{p-1} \Phi _p (V/p))\ar[u] \ar[ur]&\\
}$$
$$T \to W \quad , \quad V \to pW.$$
La fibre en $p$ de $\overline Y \to X$ correspond au morphisme d'alg\`ebres
$$\xymatrix{\bF _p [V]/ (V^{p-1})\ar[r] & \bF _p [W] /(\Phi _p (W))\\
}$$
$$V \to pW=0$$
qui se factorise \`a travers  $$\bF _p [V]/ (V^{p-1})\to \bF _p \to \bF _p [W] /(\Phi _p (W))$$ et n'est donc pas plat. Donc $\overline Y \to X$ n'est pas un torseur.


\bibliographystyle{alpha}
\begin{thebibliographyfr}{99}
\bibitem[Ana73]{Ana} S. Anantharaman, \emph{Sch\'emas en groupes, espaces homog\`enes et espaces alg\'ebriques
sur une base de dimension 1}. Dans: Sur les groupes alg\'ebriques, pp. 5-79, Bull. Soc. Math. France, M\'em. 33, Soc. Math.
France, Paris, 1973.

\bibitem[AM69]{ArMa} M. Artin et B. Mazur, \emph{Etale homotopy}, Lecture Notes in Mathematics, vol. 100, Springer-Verlag, Berlin-New York, 1969.

\bibitem[DG70]{DG}
M. Demazure et P. Gabriel, \emph{Groupes alg\'ebriques. Tome I: G\'eom\'etrie alg\'ebrique, g\'en\'eralit\'es, groupes commutatifs.}
Avec un appendice Corps de classes local par Michiel Hazewinkel
Masson et Cie, \'Editeur, Paris;  North-Holland Publishing Co., Amsterdam, 1970.

\bibitem[SGA3]{SGA} 
M. Demazure et A. Grothendieck, \emph{Sch\'emas en groupes. I: Propri\'et\'es g\'en\'erales des sch\'emas en groupes}. S\'eminaire de G\'eom\'etrie Alg\'ebrique du Bois Marie 1962/64 (SGA 3). Lecture Notes in Mathematics, vol. 151, Springer-Verlag, Berlin-New York 1970.

\bibitem[DH18]{Duong-Hai} N. D. Duong et P. H. Hai, \emph{Tannakian duality over Dedekind rings and applications}, Math. Z.
{\bf 288} (2018), no. 3-4, 1103--1142.

\bibitem[Gas03]{GAS} C. Gasbarri, \emph{Heights of vector bundles and the fundamental group scheme of a curve}, Duke Math. J. {\bf 117} (2003), no. 2, 287--311.

\bibitem[EGA1]{EGAI} A. Grothendieck, avec la collaboration de J. Dieudonn\'e, \emph{\'El\'ements de  g\'eom\'etrie alg\'ebrique. I. Le langage des sch\'emas.} 
Inst. Hautes \'Etudes Sci. Publ. Math. No. 4 (1960).

\bibitem[EGA4]{EGAIV-3} A. Grothendieck, avec la collaboration de J. Dieudonn\'e, \emph{\'El\'ements de  g\'eom\'etrie alg\'ebrique. IV. \'Etude locale des sch\'emas et des morphismes
de sch\'emas III.} Inst. Hautes \'Etudes Sci. Publ. Math. No. 28 (1966).

\bibitem[Mat89]{Mat} H. Matsumura, \emph{Commutative ring theory}, Cambridge Studies in Advanced Mathematics, vol. 8,
Cambridge University Press, Cambridge, 1989.

\bibitem[MS13]{MS} V. B. Mehta et S. Subramanian, \emph{The fundamental group scheme of a smooth projective variety over a ring of Witt
vectors}, J. Ramanujan Math. Soc. {\bf 28A} (2013), 341--351.

\bibitem[Nor82]{Nor2} M. V. Nori, \emph{The fundamental group-scheme}, Proc. Indian Acad. Sci. Math. Sci. {\bf 91} (1982), no. 2, 73--122.

\bibitem[Ray67]{Ray} M. Raynaud, \emph{Passage au quotient par une relation d'\'equivalence plate.} Dans: Proc. Conf. Local Fields (Driebergen, 1966),
pp. 78--85, Springer, Berlin, 1967.

\bibitem[Ser68]{Ser} J-P. Serre, \emph{Groupe de Grothendieck des sch\'emas en groupes r\'eductifs d\'eploy\'es}, Inst. Hautes \'Etudes Sci. Publ. Math. No. 34 (1968), 37--52.

\bibitem[SP]{ST} The Stack Project Authors, \emph{The Stacks Project}. \href{https://stacks.math.columbia.edu/}{https://stacks.math.columbia.edu/}

\bibitem[WW80]{WWO}  W.C. Waterhouse et B. Weisfeiler, \emph{One-dimensional affine group schemes}, J. Algebra {\bf 66} (1980), 550--568.

\end{thebibliographyfr}

\end{document}